\documentclass[11pt,reqno]{amsart}
\usepackage{amsmath}
\usepackage{amssymb, bm}
\usepackage{amsthm}
\usepackage{enumerate}
\usepackage[usenames]{color}
\usepackage{eepic,epic}
\usepackage{url} 
\usepackage{algorithm,algpseudocode}
\usepackage{makecell}

\usepackage{epstopdf}
\usepackage{multirow}
\usepackage{array}
\usepackage{graphicx}

\newtheorem{theorem}{Theorem}[section]

\newtheorem{lemma}[theorem]{Lemma}
\newtheorem{prop}[theorem]{Proposition}

\theoremstyle{definition}

\theoremstyle{remark}

\numberwithin{equation}{section}
\newtheorem{ass}{Assumption}

\newcommand{\ep}{\epsilon}

\begin{document}
\title[]
{Convergence property of the Quantized Distributed Gradient descent with constant stepsizes and an effective strategy for the stepsize selection} 

\author{Woocheol Choi, Myeong-su Lee}

\address{}
\email{ }



\maketitle

\begin{abstract} In this paper, we establish new convergence results for the quantized distributed gradient descent and suggest a novel strategy of choosing the stepsizes for the high-performance of the algorithm.  Under the strongly convexity assumption on the aggregate cost function and the smoothness assumption on each local cost function, we prove   the algorithm converges exponentially fast to a small neighborhood of the optimizer whose radius depends on the stepsizes. Based on our convergence result, we suggest an effective selection of stepsizes which repeats diminishing the stepsizes after a number of specific iterations. Both the convergence results and the effectiveness of the suggested stepsize selection are also verified by the numerical experiments. 
\end{abstract}

\section{Introduction}



We are concerned with the distributed optimization in which there are $N$ agents (or players) connected by a network system characterized by a graph $\mathcal{G}$ and the objective function $f$ is composed of a sum of local functions $f_i$:
\begin{align*}
	\min\limits_{x\in\mathbb{R}^n}f(x):=\sum_{i=1}^Nf_i(x).
\end{align*}
Here, the local function $f_i$ is known only by the $i$-th agent. Each agent can exchange data with other agents of its neighborhood connected through a network. Under this setting,   the agents aim to  optimize the global objective function $f$ in a collaborative way. There has been significant interest in distributed optimization algorithms since the problem appears in various real-world applications such as the distributed control \cite{BCM,CYRC}, signal processing \cite{Boyd Gossip,QT}, and machine learning problems \cite{BCN,FCG,RB}.

Over the lasd decades, various distributed algorithms have been proposed in the literature, such as alternating direction method of multipliers (ADMM) \cite{MJ,SLYW}, distributed dual averaging \cite{FMGP,SJ}, distributed gradient descent(DGD) \cite{CK,NO,NO2}, and distributed newton method \cite{MLR}. We refer to \cite{NOR,TYWY} for a survey on the distributed optimization. 

These distributed algorithms involve communication steps of the agents. In early works, the communication was usually assumed to be implemented perfectly. However, communication bandwidth or capacity is limited in practice, and thus the data is  required to be quantized for transmittion. As the amount of data  has increased recently in various fields, including machine learning, it has become more important to develop communication-efficient distributed algorithms with evaluating their performance and efficiency (see \cite{BT,MEF}).  With this motivation, various distributed algorithms with quantized communications have been proposed and studied in the literature. We refer to \cite{DMR,DMR2,EPBI,LCWH,LHWC,NOOT2,NOOT,RN,RN2,RMHP,TMHP,YH,ZYX} for distributed algorithms involving the quantization. We also refer to  \cite{DBS,DB,HCT,SY,WYLWS} regrading the delay issue in the communication.

 In this paper, we are interested in the gradient based algorithms involving the quantized communication. The works \cite{DMR,LCWH,NOOT} studied the following distributed gradient method with quantization:
\begin{equation}\label{eq-1-0}
	x_i(k+1)=Q\bigg(\sum_{j=1}^Nw_{ij}x_{j}(k)-\alpha_k\nabla f_i(x_{i,k})\bigg)
\end{equation}
where $Q(x)$ is a quantization of the value $x$ and $w_{ij}$ is the weight for communication between agent $i$ and agent $j$. In \cite{NOOT}, the author obtained the convergence to a neighborhood when the step-sizes $\alpha_k$ is constant. The author of \cite{LCWH} considered the algorithm \eqref{eq-1-0} with a projection to bounded domain and prove d a convergence estimate for the algorithm with diminishing step-sizes. 
In \cite{DMR}, the level of quantization is set to change over iterations of the algorithm, leading to an exact convergence of the algorithm. 

The author of \cite{RMHP} proposed a variant of the DGD algorithm, so called the quantized decentralized gradient Descent(QDGD) algorithm:
  \begin{equation}
  	\begin{split}\label{QDGD}
  		x_{i}(k+1)& = (1-\ep + \ep w_{ii})x_{i}(k) + \ep \sum_{j \in N_i} w_{ij} Q(x_{j}(k)) - \alpha \ep \nabla f_i (x_{i}(k)),
  	\end{split}
  \end{equation} 
where $Q: \mathbb{R}^n \rightarrow \mathbb{R}^n$ is a quantization mapping of the following types:

\noindent \textbf{(Quantization of type 1)} A mapping $Q: \mathbb{R}^n \rightarrow \mathbb{R}^n$ satisfying
\begin{equation}\label{eq-3-52}
	\mathbb{E}\Big[ Q(x)\Big] = x\quad \textrm{and}\quad \mathbb{E}\Big[ \|Q(x) - x\|^2\Big] \leq \sigma^2,
\end{equation}
\noindent \textbf{(Quantization of type 2)} A mapping $Q: \mathbb{R}^n \rightarrow \mathbb{R}^n$ satisfying
\begin{equation}\label{eq-3-53}
	\mathbb{E}\Big[ Q(x)\Big] = x\quad \textrm{and}\quad \mathbb{E}\Big[ \|Q(x) - x\|^2 \Big] \leq \sigma^2 \|x\|^2.
\end{equation}
By choosing the parameter $\epsilon>0$ suitabley in \eqref{QDGD}, one may control the variance of the randomness introduced by quantizations. 
When each cost $f_i$ stongly convex and smooth, the work \cite{RMHP} showed that for stepsize $\alpha=c_1/T^{\delta}$ and $\epsilon=c_2/T^{3\delta/2}$ with arbitrary positive constants $c_1$ and $c_2$ and any $\delta\in(0,1/2)$, the following inequality holds:
  \begin{align*}
  	\mathbb{E}\|x_{i}(T)-x_*\|^2\leq O\left(\frac{(1+\sigma^2)}{T^{\delta}}\right),
  \end{align*}
provided that  $T\geq T_0$ for a value  $T_0>0$ determined by $\delta, c_1, c_2, f, \mathcal{G}$. 
 The work \cite{DMR2} studied the same algorithm involving a projection to bounded domain and obtained a convergence property with decreasing stepsizes for non-smooth cost functions.
  
 
The aim of this work is to establish new convergence estimates of the algorithm \eqref{QDGD}. Precisely, we will show that if the stepsizes $\alpha>0$ and $\ep>0$ are less than specific values, then the algorithm \eqref{QDGD} converges exponentially fast to a  neighborhood of the optimizer in the following sense:
\begin{align}\label{choice1}
	\begin{split}
	\mathbb{E}\|x_i (k)-{x}_*\|^2\leq&~  C_0 \max\left\{1- \frac{3\eta \alpha \epsilon}{2},1 - \frac{c\epsilon }{2}\right\}^{k} + k C_1 \max\left\{1- \frac{3\eta \alpha \epsilon}{2},1 - \frac{c\epsilon }{2}\right\}^{k-1}
\\	&+C_2\alpha^2+C_3\sigma^2\epsilon+C_4\sigma^2\frac{\epsilon}{\alpha},
\end{split}
\end{align}
where the constants $C_j >0$ for $0\leq j \leq 4$ are determined by the property of the cost functions and the initial values. The assumption on the cost functions is weakened in the sense that we only assume that the aggregate cost $f$ is stronlgy convex but do not assume the convexity of each cost $f_i$, while the previous work \cite{RMHP} assumed that either all the cost functions $f_i$ are convex or strongly convex. In addition, our result is obtained for both types of  the quantizers \eqref{eq-3-52} and \eqref{eq-3-53} in a unified way.

 The above estimate \eqref{choice1} clarifies the effect on the size of the variance $\sigma >0$ for the performance of the QDGD, which is useful for choosing an effective stepsize $\alpha >0$ and $\ep>0$ depending on the size of the variance $\sigma >0$. Based on \eqref{choice1}, we propose a strategy choosing the stepsizes that decreases the stepsizes after a finite number of iterations (see Algorithm \ref{algo}). The efficiency of the proposed stepsize selection will be also verified by the numerical experiment in Section \ref{sec-6}.

This paper is organized as follows. In Section \ref{sec-3}, we give assumptions used throughout this paper and state our convergence result. Based on the result, we propose a novel strategy for choosing the stepsizes. In Section \ref{sec-4}, we obtain two sequential estimates, which will be used in Section \ref{sec-5} to prove the main results. The numerical experiments are presented in \mbox{Section \ref{sec-6}.}

\section{Main Result}\label{sec-3}
In this section, we state our main estimates on the QDGD algorithm \eqref{QDGD}. Also we present a novel strategy for choosing the stepsizes $\alpha$ and $\epsilon$, based on the  estimates. 

Before stating the main theorem, we introduce assumptions on the cost functions and give some notations used throughout this paper. 
\begin{ass}\label{LS} For each $i\in\{1,\cdots n\}$, function 
	$f_i$ is $L_i$-smooth for some $L_i>0$, i.e., for any $x, y \in \mathbb{R}^d$ we have
	\begin{equation}\label{L-smooth}
		\| \nabla f_i(y) - \nabla f_i(x)\| \leq L_i\|y-x\|\quad \forall~x,y \in \mathbb{R}^d.  
	\end{equation} 
	We set $L = \max_{1\leq i \leq n} L_i$.
\end{ass}
\noindent We assume the strongly convexity on the total cost function.
\begin{ass}\label{sc} The function
	$f$ is $\mu$-strongly convex i.e., there exists $\mu>0$ such that 
	\begin{equation*}\label{strong}
		f(y) \geq f(x) + (y-x)\nabla f(x) + \frac{\mu}{2}\|y-x\|^2
	\end{equation*}
	for all $x,y \in \mathbb{R}^d$.
\end{ass}
\noindent The agents in the problem \eqref{QDGD} share their information of states through a communication network. It is described by an undirected graph $\mathcal{G}=(\mathcal{V},\mathcal{E})$, where each node in $\mathcal{V}$ represents an agent, and each edge $\{i,j\} \in \mathcal{E}$ means that $i$ can send messages to $j$ and vice versa. We consider a graph $\mathcal{G}$ satisfying the following assumption.
\begin{ass}\label{graph}
	The communication graph $\mathcal{G}$ is undirected and connected, i.e., there exists a path between any two agents.
\end{ass}
\noindent The mixing matrix $W=\{w_{ij}\}_{1 \leq i, j \leq N}$ consisting of the weights $w_{ij}$ in \eqref{QDGD} is related to the graph in the sense that  $w_{ij}> 0$ if $\{i,j\}\in\mathcal{E}$ and $w_{ij} = 0$ if $\{i,j\}\notin\mathcal{E}$. The following is a standard assumption  on the mixing matrix $W$. 
\begin{ass}\label{ass-1-1}
	The mixing matrix $W = \{w_{ij}\}_{1 \leq i,j \leq N}$ is doubly stochastic, i.e., $W1 =1$ and $1^T W = 1^T$. In addition, $w_{ii}>0$ for some $i \in \mathcal{V}$. 
\end{ass}
\noindent It is convenient to rewrite the algorithm \eqref{QDGD} as the vectorized form as follows:
 \begin{align}\label{QDGD2}
 	\bold{x}(k+1)=(1-\epsilon)\bold{x}(k)+\epsilon W\bold{z}(k)-\alpha\epsilon\nabla F(\bold{x}(k))
 \end{align}
 where we used the following notations:
 \begin{align*}
 	\bold{x}(k)&=(x_{1}(k),x_{2}(k),\cdots,x_{N}(k))^\top\in\mathbb{R}^{N\times n},\\
 	\bold{z}(k)&=(z_{1}(k),z_{2}(k),\cdots,z_{N}(k))^\top\in\mathbb{R}^{N\times n},\\
 	\nabla F(\bold{x}(k))&=(\nabla f_1(x_{1}(k)),\nabla f_2(x_{2}(k)),\cdots,\nabla f_N(x_{N}(k)))^\top\in\mathbb{R}^{N\times n}.
 \end{align*}
 We also set the following notations:
 \begin{align*}
 	\bold{x}_*&=(x_*,\cdots,x_*)^\top\in\mathbb{R}^{N\times n}\\
 	\bar{\bold{x}}(k)&=(\bar{x}(k),\cdots,\bar{x}(k))\in\mathbb{R}^{N\times n},\\
 	\bar{\bold{z}}(k)&=(\bar{z}(k),\cdots,\bar{z}(k))\in\mathbb{R}^{N\times n}
 \end{align*}
 where $\bar{x}(k)=\frac{1}{N}\sum_{i=1}^Nx_{i}(k)$ and $\bar{z}(k)=\frac{1}{N}\sum_{i=1}^N z_{i}(k)$. We define $\beta$ by the spectral norm of the matrix $W - \frac{1}{n} 1 1^T$, and use the following constants
\begin{equation*}
\begin{split}
C:=\|\nabla F(\mathbf{x}_*)\|/L,\qquad \eta:=\frac{\mu L}{\mu+L},\qquad
		c:=\frac{1-\beta}{2}.
		\end{split}
		\end{equation*}
	We consider the quantization $Q$ of type 1 satisfying \eqref{eq-3-52} and that of type 2 satisfying \eqref{eq-3-53}.
Now we are ready to state the  main results of this paper.
\subsection{Theoretical estimates}
		We first show that the sequence of \eqref{QDGD2} are uniformly bounded under suitable assumptions on the stepsizes. 
\begin{theorem}\label{lem1-8} Suppose that Assumptions 1-4 holds. Then the sequence $\{\mathbf{x}(k)\}_{k \geq 0}$ of \eqref{QDGD2}  is bounded for suitable stepsizes as stated in the following.
\mbox{~}
\begin{enumerate}
\item For quantization  of type 1, we set $q =\frac{5\eta^2}{L(\eta+2L)}$ and define $R>0$ by 
 	\begin{equation}\label{eq-2-R1}
 	  R=\max\Big\{\frac{4N\sigma^2}{\eta},~ 
 	\frac{4}{qc}\Big[ \frac{3L^2 \alpha_0^2}{c} + 4N\ep_0 \sigma^2\Big],~ \|\bar{\bold{x}}(0)-\bold{x}_*\|^2,~\frac{\|\bar{\bold{x}}(0)-\bold{x}(0)\|^2}{q} \Big\}.
 	\end{equation}
 	Assume that $\epsilon \alpha \leq \frac{2}{\mu +L}$, $\alpha\leq \min\left\{\frac{c}{L},\sqrt{\frac{qc^2}{12L^2}}\right\}$ and  $\epsilon\leq\min\left\{\frac{1}{2c},\alpha \right\}$. Then we have
 	\begin{equation*}
 		\mathbb{E} \|\bar{\bold{x}}(k) -\bold{x}(k)\|^2 \leq qR \quad \textrm{and}\quad \mathbb{E}\|\bar{\bold{x}}(k) - \bold{x}_*\|^2 \leq R
 	\end{equation*}
 	for all $k \geq 0$.
 	\item For quantization of type 2, we set $q = \frac{\eta^2}{L(\eta + 2L)}$ and define $R>0$ by
 	 	\begin{equation}\label{eq-2-R2}
 	  R=\max\Big\{\frac{8\sigma^2\|\bold{x}_*\|^2}{\eta},~ 
 	\frac{4}{qc}\Big[ \frac{3L^2 \alpha_0^2}{c} + 8\ep_0 \sigma^2\|\bold{x}_*\|^2\Big],~ \|\bar{\bold{x}}(0)-\bold{x}_*\|^2,~\frac{\|\bar{\bold{x}}(0)-\bold{x}(0)\|^2}{q} \Big\}.
 	\end{equation}
 	 	Assume that $\epsilon \alpha \leq \frac{2}{\mu +L}$, $\alpha \leq \frac{c}{L}$, $\epsilon\leq\min\left\{\frac{1}{2c},\alpha, \frac{\eta\alpha}{2\sigma^2}\right\}$ and
 	 	\begin{equation*}
\frac{3}{c}L^2 \alpha^2 + 8 \epsilon \sigma^2 \leq \frac{qc}{4}.
 	 	\end{equation*}
 	 	Then we have
 	\begin{equation*}
 		\mathbb{E} \|\bar{\bold{x}}(k) -\bold{x}(k)\|^2 \leq qR \quad \textrm{and}\quad \mathbb{E}\|\bar{\bold{x}}(k) - \bold{x}_*\|^2 \leq R
 	\end{equation*}
 	for all $k \geq 0$.
 	\end{enumerate}
 \end{theorem}
We mention that the above result is obtained assuming the strongly convexity only for the total cost $f$. The above boundedness result might be guaranteed for a wider ranges of the stepsizes $\ep >0$ and $\alpha>0$ if we impose more conditions on the cost functions, e.g., convexity of each local cost $f_j$. This is also observed in the numerical simulations.  Having this perspective, and for simplicity of the statements for the convergence results, we propose the following assumption.  
\begin{ass}\label{ass-2} There exist  $R>0$ and $q>0$ such that
 \begin{equation*}
 		\mathbb{E} \|\bar{\bold{x}}(k) -\bold{x}(k)\|^2 \leq qR \quad \textrm{and}\quad \mathbb{E}\|\bar{\bold{x}}(k) - \bold{x}_*\|^2 \leq R
 	\end{equation*}
 	for all $k \geq 0$.
 	\end{ass}
 
Under this assumption, we obtain the following convergence estimate of the algorithm for the constant stepsizes.
 \begin{theorem}\label{thm-1} Suppose that Assumptions 1-5 hold and assume that $\ep  \alpha  \leq \frac{2}{\mu +L}$. 
\begin{enumerate}
	\item For quantization  of type 1,  we have the following estimate
 \begin{equation*}
\begin{split}
&\mathbb{E}\|\bar{\bold{x}}(k)-\bold{x}_*\|^2
\\
  &\leq\left(1- \frac{3\eta \alpha \epsilon}{2}\right)^k\mathbb{E}\|\bar{\bold{x}}(0)-\bold{x}_*\|^2
  \\&\quad+ kL \alpha \epsilon \Big( \frac{1}{4}+ \frac{L}{2\eta}\Big)\max\left\{1- \frac{3\eta \alpha \epsilon}{2},1 - \frac{c\epsilon }{2}\right\}^{k-1}\mathbb{E}\|\bar{\bold{x}}(0)-\bold{x}(0)\|^2\\
 &\quad+\frac{4L}{3\eta c}\Big( \frac{1}{4}+ \frac{L}{2\eta}\Big)\frac{2}{c}\left[\frac{3L^2 \alpha^2}{c} (R  + C^2)+4N\epsilon\sigma^2\right] +\frac{2N\epsilon\sigma^2}{3\eta\alpha}
\end{split}
\end{equation*}
and
\begin{align*}
&\mathbb{E}\|\bar{\bold{x}}(k)-\bold{x}(k)\|^2  \\
&\leq \left(1-\frac{c\ep}{2}\right)^k \mathbb{E}\|\bar{\bold{x}}(0)-\bold{x}(0)\|^2 + \frac{2}{c}\Big[ \frac{3L^2 \alpha^2}{c}(R  +C^2) + 4N \ep \sigma^2\Big].
\end{align*}
\item For quantization  of type 2, we have the following estimate
\begin{align*}
	\mathbb{E}\|\bar{\bold{x}}(k)-\bold{x}_*\|^2 &\leq  \left(1- \frac{3\eta \alpha \epsilon}{2}\right)^k\mathbb{E}\|\bar{\bold{x}}(0)-\bold{x}_*\|^2\\
	&\quad+ kL \alpha \epsilon \Big( \frac{1}{4}+ \frac{L}{2\eta}\Big)\max\left\{1- \frac{3\eta \alpha \epsilon}{2},1 - \frac{c\epsilon }{2}\right\}^{k-1}\mathbb{E}\|\bar{\bold{x}}(0)-\bold{x}(0)\|^2\\
	&\quad+\frac{4L}{3\eta c}\Big( \frac{1}{4}+ \frac{L}{2\eta}\Big) \left[\frac{3L^2 \alpha^2}{c} (R  + C^2)+8\epsilon\sigma^2(R+\|\bold{x}_*\|^2)\right] \\
	&\quad+\frac{4(R+\|\bold{x}_*\|^2)\epsilon\sigma^2}{3\eta\alpha},
\end{align*}
and
\begin{align*}
	&\mathbb{E}\|\bar{\bold{x}}(k)-\bold{x}(k)\|^2\\
	&\leq \left(1 - \frac{c\epsilon }{2}\right)^k \mathbb{E}\|\bar{\bold{x}}(0)-\bold{x}(0)\|^2+\frac{2}{c}\left[\frac{3L^2 \alpha^2}{c} (R  + C^2) + 8\epsilon\sigma^2(R+\|\bold{x}_*\|^2)\right].\\
\end{align*}
\end{enumerate}
\end{theorem}
The above convergene results imply that the sequence $\{x_i (k)\}_{k \geq 0}$ converges exponentially fast to a neighborhood of the opitmizer $x_*$ with radius depending on the stepsizes and the variance $\sigma >0$ of the quantization. From the above estimates, we may guess that the sequence $\{x_i (k)\}_{k \geq 0}$ would coverge to the optimizer $x_*$ if we repeat  adjusting the stepsizes suitably after a finite number of iterations with fixed stepsizes. Having this motivation, we propose an effective strategy for adjusting the stepsizes in the following subsection.

\subsection{Strategy for stepsize selection}
With focusing on the effects of the sizes for the stepsizes $\epsilon$, $\alpha$ and the quantization level $\sigma$, we may add the two estimates in Theorem \ref{thm-1} to find the following estimate
\begin{align}\label{choice}
	\begin{split}
	\mathbb{E}\|\bold{x}(k)-\bold{x}_*\|^2\leq& \max\left\{1- \frac{3\eta \alpha \epsilon}{2},1 - \frac{c\epsilon }{2}\right\}^{k}
	\mathbb{E}\|\bold{x}(0)-\bold{x}_*\|^2
\\
& + kL \alpha \epsilon \Big( \frac{1}{4}+ \frac{L}{2\eta}\Big)\max\left\{1- \frac{3\eta \alpha \epsilon}{2},1 - \frac{c\epsilon }{2}\right\}^{k-1}\mathbb{E}\|\bar{\bold{x}}(0)-\bold{x}(0)\|^2
\\	&+C_1\alpha^2+C_2\sigma^2\epsilon+C_3\sigma^2\frac{\epsilon}{\alpha},
\end{split}
\end{align}
where  $C_i$, $(i=1,\cdots,3)$, are constants depending only on the cost functions $f_i$ and the communication graph $\mathcal{G}$. 
In the above estimate, we observe that the first two terms involve the following decaying weight:
\begin{equation*}
\max\left\{1- \frac{3\eta \alpha \epsilon}{2},1 - \frac{c\epsilon }{2}\right\}^{k}.
\end{equation*}
On the other hand, the remaining three terms $C_1\alpha^2+C_2\sigma^2\epsilon+C_3\sigma^2\frac{\epsilon}{\alpha}$ are fixed when the stepsizes $\alpha >0$ and $\ep>0$ are fixed. Here we observe that $C_2 \sigma^2 \ep < C_3 \sigma^2 \frac{\ep}{\alpha}$ provided that $\alpha >0$ is small enough. Therefore, it is reasonable to focus on the following two terms
\begin{equation}\label{eq-3-50}
C_1 \alpha^2 + C_3 \sigma^2 \frac{\ep}{\alpha}.
\end{equation}
Inspired by the estimate \eqref{choice}, we repeat iterating the algorithm \eqref{QDGD} with fixed stepsizes for a finite number of iterations and then diminishing the stepsizes by a certain rule. The rule of our strategy is divided into three phases as described in the below.

\medskip

\noindent \textbf{(Phase 1)} For the fixed stepsizes $\alpha =\alpha_0 >0$ and $\ep=\ep_0 >0$, we perform the algorithm up to $M$ times for some $M \in \mathbb{N}$ so that the decaying part  in the estimate \eqref{choice}
\begin{equation*}
\begin{split}
&\max\left\{1- \frac{3\eta \alpha \epsilon}{2},1 - \frac{c\epsilon }{2}\right\}^{M} \mathbb{E}\|\bold{x}(0)-\bold{x}_*\|^2
\\
&\quad + M L \alpha \epsilon \Big( \frac{1}{4}+ \frac{L}{2\eta}\Big)\max\left\{1- \frac{3\eta \alpha \epsilon}{2},1 - \frac{c\epsilon }{2}\right\}^{M-1}\mathbb{E}\|\bar{\bold{x}}(0)-\bold{x}(0)\|^2
\end{split}
\end{equation*}
becomes small enough. \\

After the $M$-iterations, we diminish the stepsizes $\alpha$ and $\epsilon$ in a proper way to make the last three terms in \eqref{choice} smaller. Since reducing the stepsizes results in slowing down the convergence speed of the exponential part in \eqref{choice}, it is important to adjust the stepsize carefully. In the error term \eqref{eq-3-50}, we note that if the following inequality holds:
\begin{equation}\label{trigger}
	C_1 \alpha^2 \gg C_3 \sigma^2 \frac{\ep}{\alpha},
\end{equation}
then it is reasonable to adjust only the stepsize $\alpha$ for reducing the error term. Having this in mind, we adopt different update rules depending on whether the inequality \eqref{trigger} holds or not. Since we usually do not know the exact values of $C_1$ and $C_3$ in real applications, we choose a value $K$ instead of $\frac{C_1}{C_3}$ and check the following condition instead of \eqref{trigger}:
\begin{equation}\label{eq-3-31}
	K\alpha^2 > \sigma^2 \frac{\ep}{\alpha}.
\end{equation}

For the update rules, we use the notation $\alpha_{new}$ and $\ep_{new}$ to denote the new stepsizes from the update and denote by $\alpha_{prev}$ and $\ep_{prev}$ to denote the latest stepsizes before the update.

\noindent \textbf{(Phase 2)} We choose a sufficiently large integer $J= J_{prev}$  such that 
\begin{equation}\label{eq-3-51}
\max\left\{ 1- \frac{3 \eta \alpha_{prev} \ep}{2},~1 - \frac{c\ep}{2}\right\}^{J} < \frac{1}{2}
\end{equation}
 and we perform the algorithm for $J$-iterations. 
If the inequality \eqref{eq-3-31} holds, then we update only the stepsize $\alpha$ and the value $J$ by the following way 
\begin{align*}
	\alpha_{new} \longleftarrow\frac{\alpha_{prev}}{2}, \qquad J_{new}\longleftarrow 2J_{prev}.
\end{align*}
With these updated $\alpha_{new}$ and $J_{new}$, we perform the algorithm for  $J=J_{new}$ iterations and repeat this process as long as the inequality \eqref{eq-3-31} holds.

The reason why we adjust $J$ by $2J$   is that    choosing smaller $\alpha$ makes the decaying speed slower in the exponential decaying term \eqref{eq-3-51}. More precisely, we observe that
\begin{equation*}
\Big\{ 1 - \frac{3 \eta \alpha_{new} \ep}{2}\Big\}^{2} = \Big\{ 1 - \frac{3 \eta \alpha_{prev} \ep}{4}\Big\}^{2} \simeq 1 - \frac{3\eta \alpha_{prev} \ep}{2},
\end{equation*}
and so 
\begin{equation}
\max\left\{ 1- \frac{3 \eta \alpha_{new} \ep}{2},~1 - \frac{c\ep}{2}\right\}^{2J} \simeq\max\left\{ 1- \frac{3 \eta \alpha_{prev} \ep}{2},~1 - \frac{c\ep}{2}\right\}^{J}.
\end{equation}


\noindent \textbf{(Phase 3)} If the inequality \eqref{eq-3-31} violates, i.e., 
\begin{equation*}
K\alpha^2 \leq \sigma^2 \frac{\ep}{\alpha},
\end{equation*}
then we begin to update both the stepsizes $\alpha >0$ and $\ep>0$ in a way that
\begin{equation*}
\alpha_{new} \longleftarrow \frac{\alpha_{prev}}{\sqrt{2}},\qquad \ep_{new} \longleftarrow \frac{\ep_{prev}}{2 \sqrt{2}}, \qquad J_{new} \longleftarrow4J_{prev}.
\end{equation*}
Notice that the value $C_2 \alpha^2 + C_4 \sigma^2 \frac{\ep}{\alpha}$ becomes half for new $\alpha$ and $\ep$, i.e., 
\begin{equation*}
c_2 \alpha_{new}^2 + c_4 \sigma^2 \frac{\ep_{new}}{\alpha_{new}} = \frac{1}{2} \Big(c_2 \alpha_{prev}^2 + c_4 \sigma^2 \frac{\ep_{prev}}{\alpha_{prev}}\Big),
\end{equation*}
and we also note that
\begin{equation*}
\left\{1 - \frac{3\eta \alpha_{new}\ep_{new}}{2}\right\}^4 = \left\{1 - \frac{3 \eta \alpha_{prev} \ep_{prev}}{8}\right\}^4\simeq 1-\frac{3\eta\alpha_{prev}\ep_{prev}}{2}.
\end{equation*}
Thus it is reasonable to quadruple the number of iterations $J$ for the updated stepizes in view of \eqref{eq-3-51}. The overall strategy  is stated in Algorithm \ref{algo}.
 \begin{algorithm}\caption{Stepsize selection for the quantized decentralized gradient descent}\label{algo}
	\begin{algorithmic}
	\State{Take a value $K>0$ and $M,J \in \mathbb{N}$. Choose initial stepsizes $\alpha(0)=\alpha >0$ and $\ep(0)=\ep>0$. Set: $ind=1$ and $\text{Trigger} =0$.}
	\For{$it = 0, 1, \cdots M-1$}
	\State{Set $\alpha (t) = \alpha$ and $\ep(t) = \ep$}
		\EndFor
       \State{Set $t_Q = M$}
	\For{$it=M,M+1,\cdots$,}
         
       \If {$\mbox{mod}(it-t_Q, J*ind)=0$}
       		\If {$K\alpha^2 < \frac{\sigma^2 \ep}{\alpha}$}
       		\State{Trigger=1}
       		\EndIf
            \If { $\text{Trigger}=0$}
            \State{Let  $\alpha=\frac{\alpha}{2}$. Set $t_Q = it$ and $ind =2 \cdot  ind$}
            \EndIf
            \If { $\text{Trigger} =1$}
            \State{Let $\alpha  =\frac{\alpha}{\sqrt{2}}$, 
         $\ep = \frac{\ep}{2\sqrt{2}}$. Set $t_Q = it$ and 
         $ind= 4\cdot ind$}
            \EndIf
        \EndIf
        \State{Set $\alpha (t) = \alpha$ and $\ep(t) = \ep$}
		\EndFor
	\end{algorithmic}
\end{algorithm} 

 \section{Sequential estimates}\label{sec-4}
 In this section, we establish two sequential estimates of $\mathbb{E}\|\bar{x}_{t+1} - x_*\|^2$ and $\mathbb{E} \|x(t+1) - \bar{x}(t+1)\|^2$. We begin with recalling the following result in \cite[Lemma 1]{PN}.
 \begin{lemma}[\cite{PN}]\label{lem-1-1}
 	Suppose the Assumption 3 and Assumption 4 hold. Then the spectral norm $\beta$ of the matrix $W - \frac{1}{n} 1 1^T $ satisfies $0<\beta <1$ and   
 	\begin{equation*}
 		\Big\| W\bold{x}-\bar{\bold{x}} \Big\|^2 \leq \beta^2 \|\bold{x}-\bar{\bold{x}}\|^2.
 	\end{equation*}
 	for any $\bold{x} = (x_1, \cdots, x_N) \in \mathbb{R}^{N\times n}$.
 \end{lemma}
 \noindent Using this lemma, we dervie a sequential estimate for the consensus error.
 \begin{prop}\label{lem-1-6} For $\alpha\leq \frac{c}{L}$ and $\epsilon\leq\frac{1}{2c}$, we have
 	\begin{equation*}
 		\begin{split}
 			&\mathbb{E}\Big[ \|\bold{x}(k+1) - \bar{\bold{x}}(k+1)\|^2|\bold{x}(k)\Big]
 			\\
 			&= (1-c\epsilon/2) \|\bold{x}(k) - \bar{\bold{x}}(k)\|^2 + \frac{3L^2 \alpha^2 \epsilon}{c} \Big( \|\bar{\bold{x}}(k) - \bold{x}_*\|^2 + C^2\Big)+ 4\beta^2 \epsilon^2 r(k),
 		\end{split}
 	\end{equation*}
 	where $C=\|\nabla F(x_*)\|/L$, and $c=(1-\beta)/2$. Also 
 		\begin{equation*}
 r(k)=	 \bigg\{\begin{array}{ll} N  \sigma^2 & \textrm{for quantization of type 1 satisfying \eqref{eq-3-52}},
 		\\
 	\sigma^2 \|\bar{\bold{x}}(k)\|^2 & \textrm{for quantization of type 2 satisfying \eqref{eq-3-53}}.
 		\end{array}
\end{equation*}
 \end{prop}
 \begin{proof}
 	
 	Multiplying $\frac{1}{N} \bold{1}\bold{1}^\top$ on both sides of \eqref{QDGD2} gives
 	\begin{equation}\label{eq-3-1}
 		\bar{\bold{x}}(k+1) = (1-\ep) \bar{\bold{x}}(k) + \epsilon \bar{\bold{z}}(k) - \frac{\alpha\epsilon}{N} \bold{1}\bold{1}^\top\nabla F(\bold{x}(k)).
 	\end{equation}
 	Combining this with \eqref{QDGD2}, we have
 	\begin{equation*}
 		\begin{split}
 			&\bold{x}(k+1) - \bar{\bold{x}}(k+1)
 			\\
 			& = (1-\epsilon) (\bold{x}(k) - \bar{\bold{x}}(k)) + \epsilon (W\bold{z}(k) - \bar{\bold{z}}(k)) - \alpha \epsilon \left(\bold{I}-\frac{1}{N} \bold{1}\bold{1}^\top\right) \nabla F(\bold{x}(k))\\
 			&= (1-\epsilon) (\bold{x}(k) - \bar{\bold{x}}(k)) + \epsilon (W\bold{x}(k) - \bar{\bold{x}}(k))
 			\\
 			&\quad + \epsilon (W(\bold{z}(k)-\bold{x}(k)) - (\bar{\bold{z}}(k)-\bar{\bold{x}}(k)))- \alpha \epsilon \left(\bold{I}-\frac{1}{N} \bold{1}\bold{1}^\top\right) \nabla F(\bold{x}(k)).
 		\end{split}
 	\end{equation*}
 	Using that $\mathbb{E}\Big[ (W(\bold{z}(k)-\bold{x}(k)) - (\bar{\bold{z}}(k)-\bar{\bold{x}}(k)))|\bold{x}(k)\Big] =0$, we find
 	\begin{equation}
 		\begin{split}\label{eq4}
 			&\mathbb{E}\Big[ \|\bold{x}(k+1) - \bar{\bold{x}}(k+1)\|^2|\bold{x}(k)\Big]
 			\\
 			&=\Big\|(1-\epsilon) (\bold{x}(k) - \bar{\bold{x}}(k)) + \epsilon (W\bold{x}(k) - \bar{\bold{x}}(k)) - \alpha \epsilon \left(\bold{I}-\frac{1}{N} \bold{1}\bold{1}^\top\right) \nabla F(\bold{x}(k))\Big\|^2
 			\\
 			&\qquad + \epsilon^2 \mathbb{E}\Big[ \|(W(\bold{z}(k)-\bold{x}(k)) - (\bar{\bold{z}}(k)-\bar{\bold{x}}(k)))\|^2|\bold{x}(k)\Big].
 		\end{split}
 	\end{equation}
We proceed to estimate the above two terms in the right hand side. First we apply Lemma \ref{lem-1-1} to deduce
 	\begin{equation*}
 		\begin{split}
 			&\Big\|(1-\epsilon) (\bold{x}(k) - \bar{\bold{x}}(k)) + \epsilon (W\bold{x}(k) - \bar{\bold{x}}(k)) - \alpha \epsilon\left(\bold{I}-\frac{1}{N} \bold{1}\bold{1}^\top\right) \nabla F(\bold{x}(k))\Big\|
 			\\
 			&\leq \Big\|(1-\epsilon) (\bold{x}(t) - \bar{\bold{x}}(k)) + \epsilon (W\bold{x}(k) - \bar{\bold{x}}(k)) \Big\| + \Big\| \alpha \epsilon \left(\bold{I}-\frac{1}{N} \bold{1}\bold{1}^\top\right) \nabla F(\bold{x}(k))\Big\|
 			\\
 			&\leq \Big\|[(1-\epsilon)I+\epsilon W] (\bold{x}(k) - \bar{\bold{x}}(k))\Big\| + \alpha \epsilon\Big\|\left(\bold{I}-\frac{1}{N} \bold{1}\bold{1}^\top\right) \nabla F(\bold{x}(k))\Big\|
 			\\
 			&\leq (1-\epsilon + \epsilon \beta) \|\bold{x}(k) - \bar{\bold{x}}(k)\| + \alpha \epsilon \Big\|\nabla F(\bold{x}(t)) \Big\|.
 		\end{split}
 	\end{equation*}
 	 Using the triangle inequality and the smoothness of $f_i$, we get
 	\begin{equation*}
 	\begin{split}
 		\Big\|\nabla F(\mathbf{x}(t))\Big\|& \leq \Big\|\nabla F(\mathbf{x}(t))-\nabla F(\overline{\mathbf{x}}(t))\Big\|+\Big\|\nabla F(\overline{\mathbf{x}}(t))-\nabla F(\mathbf{x}_*)\Big\|+\Big\|\nabla F(\mathbf{x}_*)\Big\|
 		\\
 		&\leq L\Big( \|\bold{x}(k)- \bar{\bold{x}}(k)\| + \|\bar{\bold{x}}(k) - \bold{x}_*\| + C\Big),
 		\end{split}
 	\end{equation*}
 	where $C=\|\nabla F(\mathbf{x}_*)\|/L$.
 	Inserting this into the above estimate, we find
 	\begin{equation*}
 		\begin{split}
 			&\Big\|(1-\epsilon) (\bold{x}(k) - \bar{\bold{x}}(k)) + \epsilon (W\bold{x}(k) - \bar{\bold{x}}(k)) - \alpha \epsilon\left(\bold{I}-\frac{1}{N} \bold{1}\bold{1}^\top\right) \nabla F(\bold{x}(k))\Big\|
 			\\
 			&\leq (1-\epsilon + \epsilon \beta + L \alpha\epsilon) \|\bold{x}(k) - \bar{\bold{x}}(k)\| + L \alpha \epsilon \|\bar{\bold{x}}(k) - \bold{x}_*\| +LC \alpha \epsilon
 			\\
 			&\leq (1-c\epsilon) \|\bold{x}(k) - \bar{\bold{x}}(k)\| + L \alpha \epsilon \|\bar{\bold{x}}(k) - \bold{x}_*\| +LC \alpha \epsilon,
 		\end{split}
 	\end{equation*}
 	where we have set $c:=(1-\beta)/2$ and used that for  $\alpha\leq \frac{1-\beta}{2L}$ we have
 	\begin{equation*}
 		1-\epsilon+\epsilon\beta+L\alpha\epsilon\leq1-(1-\beta)\epsilon/2.
 	\end{equation*}
 	 We square the above inequality to get
 	\begin{equation}\label{eq-3-54}
 		\begin{split}
 			&\Big\|(1-\epsilon) (\bold{x}(k) - \bar{\bold{x}}(k)) + \epsilon (W\bold{x}(k) - \bar{\bold{x}}(k)) - \alpha \epsilon\left(\bold{I}-\frac{1}{N} \bold{1}\bold{1}^\top\right) \nabla F(\bold{x}(k))\Big\|^2
 			\\
 			&\leq (1-c\epsilon )^2 \|\bold{x}(k) - \bar{\bold{x}}(k)\|^2 + L^2 \alpha^2 \epsilon^2 (\|\bar{\bold{x}}(k) - \bold{x}_*\| +C)^2
 			\\
 			&\quad + 2L  \alpha \epsilon (1-c\epsilon  )  \|\bold{x}(k) - \bar{\bold{x}}(k)\|    (\|\bar{\bold{x}}(k) - \bold{x}_*\| +C).
 		\end{split}
 	\end{equation}
 	Using Young's inequality we have
 	\begin{equation*}
 		\begin{split}
 			&2L  \alpha \epsilon  \|\bold{x}(k) - \bar{\bold{x}}(k)\|    (\|\bar{\bold{x}}(k) -\bold{x}_*\| +C) 
 			\\
 			&\leq c \epsilon\|\bold{x}(k) - \bar{\bold{x}}(k)\|^2 + \frac{L^2 \alpha^2 \epsilon}{c}(\|\bar{\bold{x}}(k) - \bold{x}_*\| +C)^2. 
 		\end{split}
 	\end{equation*}
 	Combining this with \eqref{eq-3-54}, we obtain for small $\epsilon\leq\frac{1}{2c}$,
 	\begin{equation}\label{eq-3-55}
 		\begin{split}
 			&\Big\|(1-\epsilon) (\bold{x}(k) - \bar{\bold{x}}(k)) + \epsilon (W\bold{x}(k) - \bar{\bold{x}}(k)) - \alpha \epsilon\left(\bold{I}-\frac{1}{N} \bold{1}\bold{1}^\top\right) \nabla F(\bold{x}(k))\Big\|^2
 			\\
 			&\leq ((1-c\epsilon)^2 + c\epsilon) \|\bold{x}(k) - \bar{\bold{x}}(k)\|^2 + \Big(L^2 \alpha^2 \epsilon^2 + \frac{L^2 \alpha^2 \epsilon}{c}\Big)(\|\bar{\bold{x}}(k) - \bold{x}_*\| +C)^2
 			\\
 			&\leq (1-c\epsilon/2) \|\bold{x}(k) - \bar{\bold{x}}(k)\|^2 + \frac{3L^2 \alpha^2 \epsilon}{2c}(\|\bar{\bold{x}}(k) - \bold{x}_*\| +C)^2
 			\\
 			&\leq (1-c\epsilon/2) \|\bold{x}(k) - \bar{\bold{x}}(k)\|^2 + \frac{3L^2 \alpha^2 \epsilon}{c} \Big( \|\bar{\bold{x}}(k) - \bold{x}_*\|^2 + C^2\Big).
 		\end{split}
 	\end{equation}
 	On the other hand, we have
 	\begin{align}\label{eq-3-56}
 		&\epsilon^2 \mathbb{E}\Big[ \|(W(\bold{z}(k)-\bold{x}(k)) - (\bar{\bold{z}}(k)-\bar{\bold{x}}(k)))\|^2|\bold{x}(k)\Big]\\
 		&\leq 4\epsilon^2\beta^2\mathbb{E}\Big[ \|{\bold{z}}(k)-{\bold{x}}(k)\|^2|\bold{x}(k)\Big]
 		\\
 		&\leq\bigg\{\begin{array}{ll}  4\epsilon^2\beta^2 N\sigma^2 & \textrm{for qunatization of type 1}
 		\\
 		4\epsilon^2\beta^2\sigma^2\|{\bold{x}}(k)\|^2 & \textrm{for quantization of type 2}.
 		\end{array}
 	\end{align}
 	where we used Lemma \ref{lem-1-1} in the first inequality.
 	Gathering the estimates \eqref{eq-3-55} and \eqref{eq-3-56} in \eqref{eq4}, we deduce
 	\begin{equation*}
 		\begin{split}
 			&\mathbb{E}\Big[ \|\bold{x}(k+1) - \bar{\bold{x}}(k+1)\|^2|\bold{x}(k)\Big]
 			\\
 			&= (1-c\epsilon/2) \|\bold{x}(k) - \bar{\bold{x}}(k)\|^2 + \frac{3L^2 \alpha^2 \epsilon}{c} \Big( \|\bar{\bold{x}}(k) - \bold{x}_*\|^2 + C^2\Big)+ 4\epsilon^2\beta^2r(k).
 		\end{split}
 	\end{equation*}
 	The proof is finished.
 \end{proof}
 
 Next we deduce a sequential inequality for the distance between the average of the states and the optimum $\|\bar{\bold{x}}_{k+1} -\bold{x}_*\|^2$.
 \begin{lemma}\label{lem-1-7}  Assume that $\epsilon \alpha \leq \frac{2}{\mu +L}$.
 	Then  we have
 	\begin{align*}
 		&\mathbb{E}\Big[\|\bar{\bold{x}}(k+1) - \bold{x}_*\|^2|\bold{x}(k)\Big] \\
 		&\leq \left(1-\frac{3\eta\alpha\epsilon}{2}\right)\|\overline{\bold{x}}(k)-\bold{x}_*\|^2+2L^2\alpha\epsilon\left(\frac{1}{\mu +L}+\frac{1}{\eta}\right)\|\overline{\bold{x}}(k)-\bold{x}(k)\|^2 + \ep^2 r(\sigma),
 		 	\end{align*}
 	where we have set $\eta=\mu L/(\mu+L)$.
 \end{lemma}
 \begin{proof}
 	Using \eqref{eq-3-1} we get
 	\begin{equation}\label{eq-3-3}
 		\begin{split} 
 			&\mathbb{E}\Big[\|\bar{\bold{x}}(k+1) - \bold{x}_*\|^2|\bold{x}(k)\Big]\\
 			& =\mathbb{E} \Big[\Big\| \bar{\bold{x}}(k) - \bold{x}_*  - \frac{\alpha\epsilon}{N} \bold{1}\bold{1}^\top\nabla F(\bold{x}(k)) + \ep (\bar{\bold{z}}(k) - \bar{\bold{x}}(k))\Big\|^2|x(k)\Big]
 			\\
 			& = \Big\| \bar{\bold{x}}(k) - \bold{x}_* - \frac{\alpha\epsilon}{N} \bold{1}\bold{1}^\top \nabla F(\bold{x}(k))  \Big\|^2  +\epsilon^2\mathbb{E} \Big[\|   \bar{\bold{z}}(k) - \bar{\bold{x}}(k)\|^2|\bold{x}(k)\Big],
 		\end{split}
 	\end{equation}
 	where we used $\mathbb{E}[\|\overline{\bold{z}}(k)-\overline{\bold{x}}(k)\||\bold{x}(k)]=0$.
 	We apply the triangle inequality to deduce
 	\begin{equation}\label{eq-3-2}
 	\begin{split}
 		&\Big\| \bar{\bold{x}}(k) - \bold{x}_* - \frac{\alpha\epsilon}{N} \bold{1}\bold{1}^\top \nabla F(\bold{x}(k)) \Big\| \\		
 		&\leq \Big\| \bar{\bold{x}}(k) - \bold{x}_* - \frac{\alpha\epsilon}{N} \bold{1}\bold{1}^\top\nabla F(\bar{\bold{x}}(k))\Big\|+ \Big\|\frac{\alpha\epsilon}{N} \bold{1}\bold{1}^\top\left(\nabla F(\bar{\bold{x}}(k))-\nabla F(\bold{x}(k))\right)  \Big\|.
 	\end{split}
 	\end{equation}
 	For the first term in the last line, we have
 	\begin{align*}
 		\Big\| \bar{\bold{x}}(k) - \bold{x}_* - \frac{\alpha\epsilon}{N} \bold{1}\bold{1}^\top\nabla F(\bar{\bold{x}}(k))\Big\|^2&=N\left\|\bar{x}(k)-x_*-\frac{\alpha\epsilon}{N}\sum_{i=1}^N\nabla f_i(\bar{x}({k}))\right\|^2\\
 		&=N\left\|\bar{x}(k)-x_*-\alpha\epsilon\nabla f(\bar{x}(k))\right\|^2.
 	\end{align*}
 	Here, we recall that the aggregate cost $f$ is $\mu$-strongly convex and $L$-smooth. Therefore, by applying a standard argument (see e.g., \cite{Bu}), for stepsizes satisfying $\ep \alpha \leq \frac{2}{\mu +L}$  we have
 	\begin{align*}
 		N\left\|\bar{x}(k)-x_*-\alpha\epsilon\nabla f(\bar{x}_k)\right\|^2
 		&\leq N(1-2\eta\alpha\epsilon)\|\bar{x}(k)-x_*\|^2\\
 		&= (1-2\eta\alpha\epsilon)\|\bar{\bold{x}}(k)-\bold{x}_*\|^2,
 	\end{align*}
	where $\eta=\frac{\mu L}{\mu+L}$. Using the Cauchy-Schwartz inequality and the smoothness of $f_i$, the second term of \eqref{eq-3-2} is bounded as follows:
 	\begin{equation*}
 		\begin{split}
 			\Big\|\frac{\alpha\epsilon}{N} \bold{1}\bold{1}^\top\left(\nabla F(\bar{\bold{x}}(k))-\nabla F(\bold{x}(k))\right)  \Big\|^2 & = N\ \Big\| \frac{\alpha\epsilon}{N}\sum_{i=1}^N (\nabla f_i (\bar{x}(k)) - \nabla f_i (x_{i}(k)))\Big\|^2
 			\\
 			& = \frac{(\alpha\epsilon)^2}{N}\Big\| \sum_{i=1}^N (\nabla f_i (\bar{x}(k)) - \nabla f_i (x_{i}(k)))\Big\|^2
 			\\
 			& \leq (\alpha\epsilon)^2\sum_{i=1}^N \|\nabla f_i (\bar{x}(k)) - \nabla f_i (x_{i}(k))\|^2 \\
 			& \leq (\alpha\epsilon L)^2\|\bar{\bold{x}}(k) - \bold{x}(k)\|^2.
 		\end{split}
 	\end{equation*}
 	Combining the above two estimates with \eqref{eq-3-2}, we achieve the following inequality
 	\begin{equation*}
 		\begin{split}
 			& \Big\| \bar{\bold{x}}(k) - \bold{x}_* - \frac{\alpha\epsilon}{N} \bold{1}\bold{1}^\top \nabla F(\bold{x}(k))   \Big\|^2
 			\\
 			&\leq (1-2\eta\alpha\epsilon)\|\overline{\bold{x}}(k)-\bold{x}_*\|^2+L^2\alpha^2\epsilon^2\|\overline{\bold{x}}(k)-\bold{x}(k)\|^2 +2L \alpha \epsilon\|\overline{\bold{x}}(k)-\bold{x}_*\|\|\overline{\bold{x}}(k)-\bold{x}(k)\|.
 		\end{split}
 	\end{equation*}
 	Using Young's inequality, the right hand side is bounded by
 	\begin{equation*}
 		(1-3\eta\alpha\epsilon/2)\|\overline{\bold{x}}(k)-\bold{x}_*\|^2+\Big(L^2\alpha^2\epsilon^2+\frac{2L^2\alpha\epsilon}{\eta}\Big)\|\overline{\bold{x}}(k)-\bold{x}(k)\|^2.
 	\end{equation*}
 	Since $\epsilon \alpha \leq \frac{2}{\mu +L}$, we have
 	\begin{equation*}
 		\begin{split}
 			&\Big\| \bar{\bold{x}}(k) - \bold{x}_* - \frac{\alpha\epsilon}{N} \bold{1}\bold{1}^\top \nabla F(\bold{x}(k))  \Big\|^2
 			\\
 			&\leq \left(1-\frac{3\eta\alpha\epsilon}{2}\right)\|\overline{\bold{x}}(k)-\bold{x}_*\|^2+L\alpha\epsilon\left(\frac{2L}{\mu +L}+\frac{2L}{ \eta}\right)\|\overline{\bold{x}}(k)-\bold{x}(k)\|^2.
 		\end{split}
 	\end{equation*}
 	Combining this and \eqref{eq-3-3} along with the following inequality 
 	\begin{align*}
 		\epsilon^2\mathbb{E} \Big[\|   \bar{\bold{z}}(k) - \bar{\bold{x}}(k)\|^2|\bold{x}(k)\Big]\leq \bigg\{\begin{array}{ll}N\epsilon^2\sigma^2 &\textrm{for qunatization of type 1}
 		\\
 		\epsilon^2 \sigma^2\|\bar{\bold{x}}(k)\|^2 &\textrm{for qunatization of type 2},
 		\end{array}
 	\end{align*}
 	we get
 	\begin{align*}
 		&\mathbb{E}\Big[\|\bar{\bold{x}}(k+1) - \bold{x}_*\|^2|\bold{x}(k)\Big] \\
 		&\leq \left(1-\frac{3\eta\alpha\epsilon}{2}\right)\|\overline{\bold{x}}(k)-\bold{x}_*\|^2+L\alpha\epsilon \left(\frac{2L}{\mu +L}+\frac{2L}{ \eta}\right)\|\overline{\bold{x}}(k)-\bold{x}(k)\|^2 + \epsilon^2 r(\sigma).
 	\end{align*}
 	This finishes the proof.
 \end{proof}

\section{Proofs of the main theorems}\label{sec-5}
 In this section, we give the proofs of Theorem \ref{lem1-8} and Theorem \ref{thm-1}. For this we take the expectation over the whole time from $0$ to $k$ on the inequalities of Lemma \ref{lem-1-6} and Lemma \ref{lem-1-7} to get
 \begin{equation*}
 	\begin{split}
 		&\mathbb{E}\Big[ \|\bold{x}(k+1) - \bar{\bold{x}}(k+1)\|^2\Big]
 		\\
 		&= (1-c\epsilon/2) \mathbb{E}\Big[\|\bold{x}(k) - \bar{\bold{x}}(k)\|^2\Big] + \frac{3L^2 \alpha^2 \epsilon}{c} \Big( \mathbb{E}\Big[\|\bar{\bold{x}}(k) - \bold{x}_*\|^2\Big] + C^2\Big)+ 4\epsilon^2\beta^2r(k),
 	\end{split}
 \end{equation*}
 and 
 \begin{align*}
 	\begin{split}
 		&\mathbb{E}\Big[\|\bar{\bold{x}}(k+1) - \bold{x}_*\|^2\Big] \\
 		&\leq \left(1-\frac{3\eta\alpha\epsilon}{2}\right)\mathbb{E}\Big[\|\overline{\bold{x}}(k)-\bold{x}_*\|^2\Big]+2L^2\alpha\epsilon\left(\frac{1}{\mu +L}+\frac{1}{\eta}\right)\mathbb{E}\Big[\|\overline{\bold{x}}(k)-\bold{x}(k)\|^2\Big] + \epsilon^2r(k).
 	\end{split}
 \end{align*}
For simplicity of exposition, we define the constants $A_k$ and $B_k$ for $k \geq 0$ by
\begin{equation*}
A_k = \mathbb{E}\|\bar{\mathbf{x}}(k)-\mathbf{x}_*\|^2, \quad B_k = \mathbb{E}\|\mathbf{x}(k) - \bar{\mathbf{x}}(k)\|^2.
\end{equation*}
Using these notations, the above inequalities are written as
  	\begin{equation}\label{eq-3-7} 
 			B_{k+1}  \leq \left(1 - \frac{c\epsilon }{2}\right) B_k + \frac{3L^2 \alpha^2 \epsilon}{c} (A_k + C^2) + 4\beta^2 \epsilon^2 r(k)
 			 			\end{equation}
 			and
 			\begin{equation}\label{eq-3-8}
 			A_{k+1} \leq \left(1- \frac{3\eta \alpha \epsilon}{2}\right) A_k + L \alpha \epsilon \Big( \frac{1}{4} + \frac{L}{2\eta}\Big) B_k +  \epsilon^2 r(k).
 			\end{equation}
Based on the above two inequalities, we prove the boundedness result of Theorem \ref{lem1-8} in the below.
 \begin{proof}[Proof of Theorem \ref{lem1-8}]
 	We argue by an induction. Trivially, $A_0\leq R$ and $B_0\leq qR$ by the definition of $R$. Assume that $A_k \leq R$ and $B_k \leq qR$ for some $k \geq 0$. First we complete the proof for the quantization of type 1. Using \eqref{eq-3-8}, we have
 	\begin{equation*}
 		\begin{split}
 			A_{k+1}& \leq \Big[ \Big( 1- \frac{3}{2} \eta\alpha \epsilon\Big) + q L \alpha \epsilon \Big( \frac{1}{4} + \frac{L}{2\eta}\Big) \Big] R +  N\epsilon^2 \sigma^2
 			\\
 			&= \Big[ 1 + \alpha \epsilon \Big( q L \Big( \frac{1}{4} + \frac{L}{2\eta}\Big) - \frac{3\eta}{2}\Big) \Big] R + N\epsilon^2 \sigma^2 
 			\\
 			&= \Big[ 1- \frac{\alpha \epsilon}{4}\eta\Big] R + N\epsilon^2 \sigma^2
 			  \\ 
 			& = R + \epsilon \Big[ N\epsilon \sigma^2 - \frac{\alpha}{4}\eta R\Big],
 			\end{split}
 			\end{equation*}
 			where we used the definition of $q$ in the second equality. Next we use $\epsilon \leq \alpha$ and the definition \eqref{eq-2-R1} of $R$ to deduce
 			\begin{equation*}
A_{k+1}		\leq R + \epsilon \alpha \Big[ N\sigma^2 - \frac{\eta}{4} R\Big] \leq R.
 			\end{equation*}
 	 Then, we use \eqref{eq-3-7} to estimate  $B_{k+1}$ as
 	\begin{equation*}
 		\begin{split}
 			B_{k+1} & \leq q \Big( 1- \frac{c\epsilon}{2}\Big) R + (3/c) L^2 \alpha^2 \epsilon (R + C^2) + 4N \epsilon^2 \sigma^2
 			\\
 			& = R \Big[ q- qc\epsilon /2 + (3/c) L^2 \alpha^2\epsilon  \Big] +(3/c) L^2 \alpha^2 \epsilon C^2 + 4N \epsilon^2 \sigma^2
 			\\
 			&\leq R [q-qc\epsilon /4] +(3/c) L^2 \alpha^2 \epsilon C^2 + 4N \epsilon^2 \sigma^2
 			\\
 			& = Rq - \epsilon \Big[Rq c /4 - (3/c) L^2 \alpha^2 - 4N \epsilon \sigma^2\Big]
 			\\
 			&\leq Rq,
 		\end{split}
 	\end{equation*} 
 	where the second inequality and the third inequality also hold by the definition \eqref{eq-2-R1} of $R$.
 	\
 	
 	Next we prove the result for the quantization of type 2. Using that $\|\bold{x}(k)\|^2 \leq 2\|\bold{x}(k)-\bold{x}_*\|^2+2\|\bold{x}_*\|^2$ and $2\sigma^2 \epsilon \leq \alpha \eta$, we deduce
 	\begin{equation*}
 		\begin{split}
 			A_{k+1}& \leq \Big[ \Big( 1- \frac{3}{2} \eta\alpha \epsilon\Big) + q L \alpha \epsilon \Big( \frac{1}{4} + \frac{L}{2\eta}\Big) \Big] R +   \epsilon^2 \sigma^2 \|\bold{x}(k)\|^2
 			\\
 			&\leq  \Big[ 1 + \alpha \epsilon \Big( q L \Big( \frac{1}{4} + \frac{L}{2\eta}\Big) - \frac{3\eta}{2}\Big) \Big] R + \epsilon^2 \sigma^2  (2R + 2\|\bold{x}_*\|^2)
 			\\
 			&\leq  \Big[ 1 + \alpha \epsilon \Big( q L \Big( \frac{1}{4} + \frac{L}{2\eta}\Big) - \frac{ \eta}{2}\Big) \Big] R + 2\epsilon^2 \sigma^2  \|\bold{x}_*\|^2\\
 			&= \Big[ 1- \frac{\alpha \epsilon}{4}\eta\Big] R +  2\epsilon^2 \sigma^2 \|\bold{x}_*\|^2
 			\end{split}
 			\end{equation*}
 			where we used the definition of $q = \frac{\eta^2}{L(\eta + 2L)}$ in the last equality. Then, by the relation $\epsilon \leq \alpha$ and the definition \eqref{eq-2-R2} of $R$, we obtain
 			\begin{align*}
			A_{k+1}\leq& R+ \epsilon\alpha\left[ 2\sigma^2\|\bold{x}_*\|^2-\frac{\eta}{4}R \right]\leq R.
 			\end{align*}
 	Finally, we use \eqref{eq-3-8} to estimate  $B_{k+1}$ as
 	\begin{equation*}
 		\begin{split}
 			B_{k+1} & \leq q \Big( 1- \frac{c\epsilon}{2}\Big) R + (3/c) L^2 \alpha^2 \epsilon (R + C^2) + 4\epsilon^2 \sigma^2 \|\bold{x}\|^2
 			\\
 			& = R \Big[ q- qc\epsilon /2 + (3/c) L^2 \alpha^2\epsilon   + 8\epsilon^2 \sigma^2\Big] +(3/c) L^2 \alpha^2 \epsilon C^2 + 8\epsilon^2 \sigma^2 \|\bold{x}_*\|^2
 			\\
 			&\leq R [q-qc\epsilon /4] +(3/c) L^2 \alpha^2 \epsilon C^2 + 8\epsilon^2 \sigma^2 \|\bold{x}_*\|^2
 			\\
 			& = Rq - \epsilon \Big[Rq c /4 - (3/c) L^2 \alpha^2 - 8\epsilon \sigma^2 \|\bold{x}_*\|^2]
 			\\
 			&\leq Rq,
 		\end{split}
 	\end{equation*} 
where we used the assumption of $\epsilon$ and $\alpha$  in the second inequality and used the definition \eqref{eq-2-R2} of $R$ in the last inequality. The proof is done.
\end{proof}

\begin{proof}[Proof of Theorem \ref{thm-1}] We first consider the quantization of type 1. Applying the assumption \ref{ass-2} to \eqref{eq-3-7}, we have
 \begin{align*}
 		B_{k+1}  \leq \left(1 - \frac{c\epsilon }{2}\right) B_k + \frac{3L^2 \alpha^2 \epsilon}{c} (R  + C^2) + 4N\epsilon^2\sigma^2
 \end{align*}
and hence
\begin{align*}
	B_{k+1}  \leq& \left(1 - \frac{c\epsilon }{2}\right)^{k+1} B_0 + \left[1+\left(1 - \frac{c\epsilon }{2}\right)+\left(1 - \frac{c\epsilon }{2}\right)^2+\cdots+\left(1 - \frac{c\epsilon }{2}\right)^k\right]\\
	&\times\left[\frac{3L^2 \alpha^2 \epsilon}{c} (R  + C^2) + 4N\epsilon^2\sigma^2\right]\\
	\leq& \left(1 - \frac{c\epsilon }{2}\right)^{k+1} B_0+\frac{2}{c}\left[\frac{3L^2 \alpha^2}{c} (R  + C^2) + 4N\epsilon\sigma^2\right].
\end{align*}
By applying the above inequality to \eqref{eq-3-8} we obtain
\begin{align*}
	A_{k+1} \leq& \left(1- \frac{3\eta \alpha \epsilon}{2}\right) A_k + L \alpha \epsilon \Big( \frac{1}{4}+ \frac{L}{2\eta}\Big)\left(1 - \frac{c\epsilon }{2}\right)^{k} B_0\\
	&+L \alpha \epsilon \Big( \frac{1}{4}+ \frac{L}{2\eta}\Big)\frac{2}{c}\left[\frac{3L^2 \alpha^2}{c} (R  + C^2) + 4N\epsilon\sigma^2\right]  +  N\epsilon^2\sigma^2.
	\end{align*}
Using this inequality iteratively, we get
\begin{equation*}
\begin{split}
A_{k+1}	\leq&  \left(1- \frac{3\eta \alpha \epsilon}{2}\right)^{k+1}A_{0} +L \alpha \epsilon \Big( \frac{1}{4}+ \frac{L}{2\eta}\Big)\sum_{j=0}^{k} \bigg(1- \frac{3\eta \alpha \epsilon}{2}\bigg)^{k-j} \bigg(1 - \frac{c\epsilon }{2}\bigg)^{j}B_0\\
	&+ \bigg[ L\alpha \ep  \Big( \frac{1}{4}+ \frac{L}{2\eta}\Big)\frac{2}{c}\left[\frac{3L^2 \alpha^2}{c} (R  + C^2)+4N\epsilon\sigma^2\right]   +N\epsilon^2\sigma^2\bigg]\sum_{j=0}^{k} \Big( 1- \frac{3 \eta \alpha \ep}{2}\Big)^j .
	\end{split}
	\end{equation*}
	From this, we easily deduce the following inequality
	\begin{equation*}
	\begin{split}
	A_{k+1} \leq&  \left(1- \frac{3\eta \alpha \epsilon}{2}\right)^{k+1}A_{0}+ (k+1)L \alpha \epsilon \Big( \frac{1}{4}+ \frac{L}{2\eta}\Big)\max\left\{1- \frac{3\eta \alpha \epsilon}{2},1 - \frac{c\epsilon }{2}\right\}^{k}B_0\\
	&+\frac{4L}{3\eta c}\Big( \frac{1}{4}+ \frac{L}{2\eta}\Big) \left[\frac{3L^2 \alpha^2}{c} (R  + C^2)+4N\epsilon\sigma^2\right] +\frac{2N\epsilon\sigma^2}{3\eta\alpha}.
\end{split}
\end{equation*}
For the quantization of type 2, we used the following inequality
\begin{align*}
	r(k)=\sigma^2\|\bold{x}\|^2\leq 2\sigma^2R+2\sigma^2\|\bold{x}_*\|^2.
\end{align*}
Then, applying the similar way for the quantization type 1, we obtain
\begin{align*}
	B_{k+1}\leq& \left(1 - \frac{c\epsilon }{2}\right)^k B_0+\frac{2}{c}\left[\frac{3L^2 \alpha^2}{c} (R  + C^2) + 8\epsilon\sigma^2(R+\|\bold{x}_*\|^2)\right],
\end{align*}
and
 \begin{align*}
 	A_{k+1} \leq&  \left(1- \frac{3\eta \alpha \epsilon}{2}\right)^kA_{0}+ (k+1)L \alpha \epsilon \Big( \frac{1}{4}+ \frac{L}{2\eta}\Big)\max\left\{1- \frac{3\eta \alpha \epsilon}{2},1 - \frac{c\epsilon }{2}\right\}^{k-1}B_0\\
 	&+\frac{4L}{3\eta c}\Big( \frac{1}{4}+ \frac{L}{2\eta}\Big)\left[\frac{3L^2 \alpha^2}{c} (R  + C^2)+4\epsilon\sigma^2(R+\|\bold{x}_*\|^2)\right] +\frac{4(R+\|\bold{x}_*\|^2)\epsilon\sigma^2}{3\eta\alpha}.
 \end{align*}
 The proof is finished. \eqref{eq-2-R2} 
 
\end{proof}
 \section{Numerical experiment}\label{sec-6}

In this section, we provide numerical experiments of the algorithm \eqref{QDGD}. We consider the cost function
\begin{equation*}
f(x)  = \frac{1}{N} \sum_{i=1}^N \|A_i x- y_i\|^2,
\end{equation*}
where $N$ is the number of agents and $A_i$ is an $m\times n$ matrix whose element is chosen randomly following the normal distribution $N(0,1)$  for each $1 \leq i \leq N$. Also, the vector  $y_i \in \mathbb{R}^m$ is generated from the normal distribution $N(0,1)$ for its elements. We have set $N=20$, $m=5$ and $n=10$.

The communication matrix $W$ is described as follows: Each two agents are linked with probability 0.4, and the weights $w_{ij}$ is defined by 
\begin{equation*}
w_{ij}= \left\{ \begin{array}{ll} 1/ \max\{\textrm{deg}(i), \textrm{deg}(j)\}&~\textrm{if}~ i \in N_i,
\\
1- \sum_{j \in N_i} w_{ij}&~\textrm{if}~ i=j,
\\
0& ~otherwise.
\end{array}
\right.
\end{equation*}  

We use the following two kinds of quantizers \cite{ACR,DMR2,RMHP}.
\begin{itemize}
	\item As a quantizer of type 1, we consider the quantizer $Q_\gamma : \mathbb{R}\rightarrow \mathbb{R}$ for $\gamma >0$ given as
	\begin{align*}
		Q_\gamma(x)=\begin{cases}
			k\gamma\qquad&\text{w.p.}\ 1-\frac{x-k\gamma}{\gamma},\\
			(k+1)\gamma\qquad&\text{w.p.}\ \frac{x-k\gamma}{\gamma}.
		\end{cases}
	\end{align*}
	where $k$ is an integer such that $k\gamma\leq x\leq (k+1)\gamma$. It applies to each coordinate if $x \in \mathbb{R}^n$. We note that this quantizer satisfies
	\begin{align*}
		\mathbb{E}\Big[ Q_\gamma(x)\Big] = x\quad \textrm{and}\quad \mathbb{E}\Big[ \|Q_\gamma(x) - x\|^2\Big] \leq \gamma^2/4,
	\end{align*}
	\item As a quantizer of type 2, we consider the quantizer $Q^s=(Q_1^s,...,Q_n^s):\mathbb{R}^n \rightarrow \mathbb{R}^n$ for $s>1$ defined by the following way. For any given $x\in\mathbb{R}^n$, we find an integer $l\in[0,s)$ such that $l/s\leq|x_i|/\|x\|\leq(l+1)/s$. Then the operator $Q_i^s$ is defined as
	\begin{align*}
		Q_i^s(x)=\|x\|\text{sign}(x_i)\xi_i(x,s)
	\end{align*}
	where $\xi_i(x,s)$ is given by
	\begin{align*}
		\xi_i(x,s)=\begin{cases}
			\frac{l}{s} \qquad&\text{w.p.}\ 1-q(|x_i|/\|x\|,s),\\
			\frac{l+1}{s} \qquad&\text{w.p.}\ q(|x_i|/\|x\|,s),
		\end{cases}
	\end{align*}
	and $q(a,s)=as-l$. This quantizer enjoys the following property
	\begin{equation*}
		\mathbb{E}\Big[ Q^s(x)\Big] = x\quad \textrm{and}\quad \mathbb{E}\Big[ \|Q^s(x) - x\|^2 \Big] \leq \min\left(\frac{n}{s^2},\frac{\sqrt{n}}{s}\right) \|x\|^2.
	\end{equation*}
\end{itemize}

We implement a series of experiments using the above two quantizers with various $\gamma$ or $s$. The aim of these experiments is twofold: 1) to support the convergence results obtained in Theorem \ref{thm-1} and 2) to show the effectiveness of our step-size selection strategy, presented in Algorithm 1, in comparison to other traditional step-size selection methods.

\subsection{Experiment 1} We perform the algorithm \eqref{QDGD} with the following choices of stepsizes:
\begin{enumerate}
	\item[Case 1]: $\alpha (k)$ and $\eta(k)$ are generated by Algorithm \ref{algo} with parameters $K=5000$, $M=1000$, and $J=500$, and the initial values $\alpha (0) = \alpha_0$ and $\ep(0) = \ep_0$.
	\item[Case 2]: $\alpha \equiv  \alpha_0$ and $\ep \equiv \ep_0$.
	\item[Case 3]: $\alpha \equiv  \alpha_0/5$ and $\ep \equiv \ep_0$.
	\item[Case 4]:  $\alpha \equiv \alpha_0/5$ and $\ep \equiv \ep_0/5$.
	\item[Case 5]:  $\alpha  (k) = \alpha_0\left(1000/(k+1000)\right)^{1/4}$ and $\ep(k)= \ep_0\left(1000/(k+1000)\right)^{3/4}$,
\end{enumerate} 
where we set $\alpha_0 = 0.01$ and $\ep_0 = 0.5$. We test the algorithm for the first type of quantization $Q_{\gamma}$ with two quantization levels $\gamma =0.01$ and $\gamma = 0.05$. We measure the error $ \sum_{i=1}^{N}\|x_i (k) -x_*\|/N$ for $k \geq 0$ and its graph in the log-scale is presented in  Figure \ref{fig2}.
 \begin{figure}[htbp]
\includegraphics[height=4.5cm, width=6.3cm]{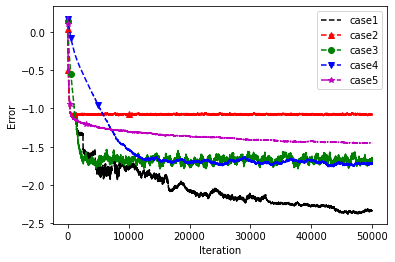}
\includegraphics[height=4.5cm, width=6.3cm]{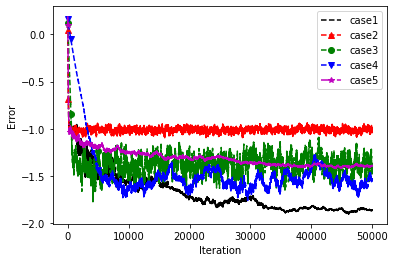} 
\vspace{-0.3cm}\caption{The graphs of the error $\log(\sum_{i=1}^{N}\|x_i (k) -x_*\|/N)$ for the QDGD  with respect to  $k \geq 0$ and stepsize $\alpha (k)$ and $\ep(k)$ given the above cases. The left figure represents results obtained with $\sigma =0.01$, and the right figure represents results obtained with $\sigma =0.05$. \label{fig2}.} 
\end{figure}

Next, we perform the same experiment with the quantization $Q^s$ for two parameters $s=10$ and $s=50$. The graphs of the errors in the log scale  are presented in Figure \ref{fig3}.
\begin{figure}[htbp]
	\includegraphics[height=4.5cm, width=6.3cm]{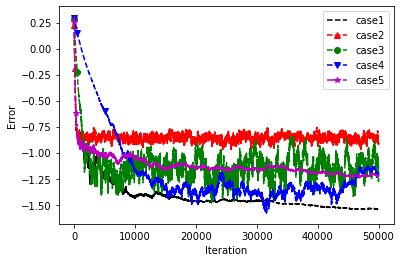} 
\includegraphics[height=4.5cm, width=6.3cm]{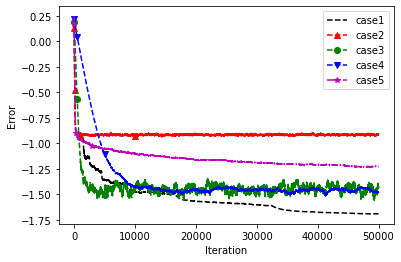} 
\vspace{-0.3cm}\caption{The graphs of the error $\log(\sum_{i=1}^{N}\|x_i (k) -x_*\|/N)$ for the QDGD  with respect to  $k \geq 0$ and stepsize $\alpha (k)$ and $\ep(k)$ for each case. The left figure represents results obtained with $s = 10$, and the right figure represents results obtained with $s=50$.}\label{fig3}
\end{figure}

In the above graphs, we figure out that the algorithm \eqref{QDGD} with the stepsizes chosen by Case 1 are superior to other cases. This confirms the effectiveness of the stepsize selection of Algorithm \ref{algo}. Also the graphs for Cases 2$-$4 supports the convergence results of Theorem \ref{thm-1}.


\subsection{Experiment 2}
In this experiment, we further show that the step of Phase 2 in Algorithm \ref{algo} is essential for the high performance as revealed in Experiment 1. For this, we compare the performance of the algorithm \eqref{QDGD} with two choices of stepsizes: 1) the original algorithm \ref{algo}, and 2) a modified version of algorithm \ref{algo} which excludes Phase 2 and utilizes only Phases 1 and 3.  In the experiment, we use the same parameters  $K, M,J,\alpha_0,$ and $\ep_0$ as in Case 1 of Experiment 1. We again measure the error $\log(\sum_{i=1}^{N}\|x_i (k) -x_*\|/N)$ for the comparison of the performance.

For the quantization of type 1, we set the quantization level as $\gamma=0.01$ and the result is  presented in  Figure \ref{fig4}.   It can be clearly observed that the algorithm \eqref{QDGD} with the  stepsize algorithm given by the original algorithm \ref{algo} which includes the step of Phase 2, outperforms the case with the modified one omitting the step of Phase 2. It emphasizes the effectiveness of Phase 2 in the stepsize selection. 
For the quantization  $Q^s$ of type 2, we set the quantization level as $s=50$ and perform the same experiment. The result given in Figure \ref{fig5} also shows that the performance with the original algorithm \ref{algo} is much better. 



 \begin{figure}[htbp]
	\includegraphics[height=4.5cm, width=6.3cm]{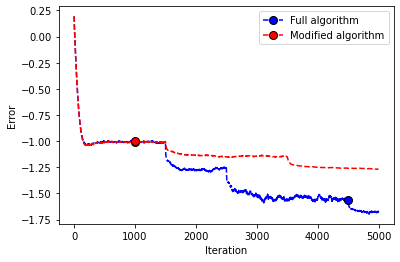}
	\includegraphics[height=4.5cm, width=6.3cm]{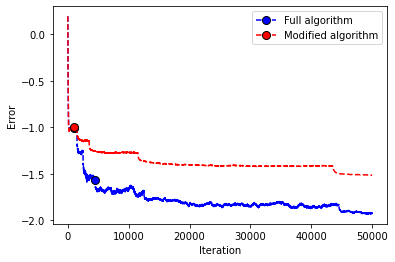} 
	\vspace{-0.3cm}\caption{\label{fig4} Comparison between the full version of the algorithm \ref{algo} and the modified algorithm. The markers on the graphs represent the points where the phase changes. The left graph displays the graph of the logarithmic error from $k=0$ to $k=5000$, while the right graph extends the range to $k=50000$.} 
\end{figure}

 \begin{figure}[htbp]
	\includegraphics[height=4.5cm, width=6.3cm]{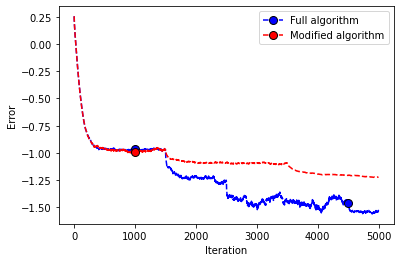}
	\includegraphics[height=4.5cm, width=6.3cm]{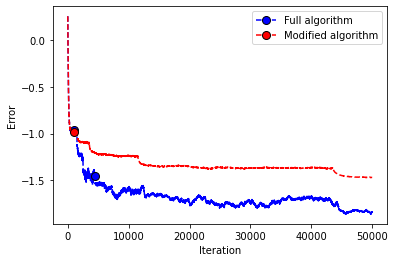} 
	\vspace{-0.3cm}\caption{\label{fig5} Same experiments as in Figure \ref{fig4} but using the second-type quantizer $Q^s$ with $s=50$. We see that the full algorithm still gives a better performance than the modified one.} 
\end{figure}

\end{document}